\documentclass[a4paper,reqno]{amsart}

\usepackage{amssymb,upref}
\usepackage[mathcal]{euscript}
\usepackage[cmtip,all]{xy}

\newcommand{\bydef}{:=}

\newcommand{\id}{\mathrm{id}}

\newcommand{\cA}{\mathcal{A}}

\newcommand{\cH}{\mathcal{H}}

\newcommand{\cJ}{\mathcal{J}}

\newcommand{\ZZ}{\mathbb{Z}}

\newcommand{\FF}{\mathbb{F}}
\newcommand{\KK}{\mathbb{K}}

\DeclareMathOperator{\Grp}{\mathrm{Grp}}
\DeclareMathOperator{\AlgF}{\mathrm{Alg_{\FF}}}

\DeclareMathOperator{\End}{\mathrm{End}}

\DeclareMathOperator{\Aut}{\mathrm{Aut}}
\DeclareMathOperator{\AAut}{\mathbf{Aut}}

\newcommand{\frsl}{{\mathfrak{sl}}}

\newcommand{\SL}{\mathrm{SL}}

\newcommand{\GGL}{\mathbf{GL}}
\newcommand{\SSL}{\mathbf{SL}}

\newcommand{\SSP}{\mathbf{Sp}}

\newcommand{\subo}{_{\bar 0}}
\newcommand{\subuno}{_{\bar 1}}

\newcommand{\Cm}{\mathsf{C}}

\renewcommand{\Aut}{\mathsf{Aut}}
\renewcommand{\End}{\mathsf{End}}

\renewcommand{\id}{\mathsf{id}}

\newcommand{\Km}{\mathsf{K}}
\newcommand{\Gm}{\mathsf{G}}
\newcommand{\Jm}{\mathsf{J}}

\newtheorem{theorem}{Theorem}
\newtheorem{proposition}[theorem]{Proposition}

\newtheorem{corollary}[theorem]{Corollary}

\theoremstyle{definition}
\newtheorem{df}[theorem]{Definition}

\newtheorem{remark}[theorem]{Remark}

\numberwithin{theorem}{section}
\numberwithin{equation}{section}

\begin{document}

\title{On Kac's Jordan superalgebra}

\author[A.S.~C\'ordova-Mart{\'\i}nez]{Alejandra S.~C\'ordova-Mart{\'\i}nez${}^\star$}
\address{Departamento de Matem\'{a}ticas
 e Instituto Universitario de Matem\'aticas y Aplicaciones,
 Universidad de Zaragoza, 50009 Zaragoza, Spain}
\email{sarina.cordova@gmail.com}
\thanks{${}^\star$ Supported by the Spanish Ministerio de Econom\'{\i}a y 
Competitividad---Fondo Europeo de Desarrollo Regional (FEDER) MTM2013-45588-C3-2-P, 
and by the Diputaci\'on General de Arag\'on---Fondo Social Europeo 
(Grupo de Investigaci\'on de \'Algebra). A.S.~C\'ordova-Mart{\'\i}nez also acknowledges support from Consejo Nacional de Ciencia y Tecnolog{\'\i}a
(CONACyT, M\'exico) through grant 420842/262964.}

\author[A.~Darehgazani]{Abbas Darehgazani${}^\diamond$}
\address{Department of Mathematics\\ University of Isfahan\\Isfahan, Iran,
P.O.Box: 81746-73441 and School of Mathematics, Institute for
Research in Fundamental Sciences (IPM), P.O. Box: 19395-5746}
\email{abbas.darehgazani@gmail.com}
\thanks{${}^\diamond$ This research was in part supported by a grant from IPM (No. 94160221) and partially carried out in
IPM-Isfahan Branch.}

\author[A.~Elduque]{Alberto Elduque${}^\star$}
\address{Departamento de Matem\'{a}ticas
 e Instituto Universitario de Matem\'aticas y Aplicaciones,
 Universidad de Zaragoza, 50009 Zaragoza, Spain}
\email{elduque@unizar.es}



\date{January 23, 2017}

\begin{abstract}
The group-scheme of automorphisms of the ten-dimensional exceptional Kac's Jordan superalgebra is shown to be isomorphic to the semidirect product of the direct product of two copies of $\SSL_2$ by the constant group scheme $\Cm_2$.

This is used to revisit, extend, and simplify, known results on the classification of the twisted forms of this superalgebra and of its gradings.
\end{abstract}

\maketitle

Kac's ten-dimensional superalgebra $\Km_{10}$ is an exceptional Jordan superalgebra which appeared for the first time in Kac's classification \cite{Kac} of the finite dimensional simple Jordan superalgebras over an algebraically closed field of characteristic $0$. It was constructed by Lie theoretical terms from a $3$-grading of the exceptional simple Lie superalgebra $F(4)$.

A more conceptual definition was given in \cite{BE02} over an arbitrary field $\FF$ of characteristic not two (\emph{these assumptions on the ground field will be kept throughout the paper}), using the three-dimensional Kaplansky superalgebra $\Km_3$. The even part $(\Km_3)\subo$ is a copy of the ground field $\FF$: $(\Km_3)\subo=\FF a$, with $a^2=a$; while the odd part is a two-dimensional vector space $W$ endowed with a nonzero skew-symmetric bilinear form $(\,\mid\,)$. The multiplication is determined as follows:
\[
a^2=a,\quad av=\frac{1}{2}v=va,\quad vw=(v\vert w)a,
\]
for any $v,w\in W$. We extend $(\,\mid\,)$ to a supersymmetric bilinear form on $\Km_3$: $(\Km_3)\subo$ and $(\Km_3)\subuno$ are orthogonal, with $(a\vert a)=\frac{1}{2}$.

Then $\Km_{10}=\FF 1\oplus\bigl(\Km_3\otimes_\FF\Km_3\bigr)$, with $(\Km_{10})\subo=\FF1\oplus\bigl(\Km_3\otimes_\FF\Km_3\bigr)\subo=\FF 1\oplus\FF(a\otimes a)\oplus (W\otimes W)$, and $(\Km_{10})\subuno=\bigl(\Km_3\otimes_\FF\Km_3\bigr)\subuno=(W\otimes a)\oplus (a\otimes W)$. The multiplication is given by imposing that $1$ is the unity, and for homogeneous elements $x,y,z,t\in\Km_3$,
\begin{equation}\label{eq:mult_K10}
(x\otimes y)(z\otimes t)=(-1)^{yz}\Bigl(xz\otimes yt-\frac{3}{4}(x\mid z)(y\mid t)1\Bigr).
\end{equation}

If the characteristic is $3$, then $\Km_9\bydef \Km_3\otimes_{\FF} \Km_3$ is a simple ideal in $\Km_{10}$. Otherwise $\Km_{10}$ is simple.

\smallskip

It must be remarked that an `octonionic' construction of $\Km_{10}$ has been given in \cite{RZ_K10_oct}.

\smallskip

Using the construction above of $\Km_{10}$ in terms of $\Km_3$, the group of automorphisms $\Aut\bigl(\Km_{10}\bigr)$ was computed in \cite{ELS07}. This was used in \cite{CDM10} to classify gradings on $\Km_{10}$ over algebraically closed fields of characteristic zero.

\smallskip

Our goal here is to compute the group-scheme of automorphisms of $\Km_{10}$, and to use it to revisit, extend, and simplify drastically, the known results on twisted forms and on gradings on $\Km_{10}$. 

Here we follow the functorial approach (see, for instance, \cite{Wat}). An affine group-scheme over $\FF$ is a representable group-valued functor defined on the category of unital, commutative, associative algebras over $\FF$: $\AlgF$. Thus the group-scheme $\AAut\bigl(\Km_{10}\bigr)$ is the functor $\AlgF\rightarrow \Grp$ that takes any object $R$ in $\AlgF$ to the group $\Aut_R\bigl(\Km_{10}\otimes_\FF R)$ (the group of automorphisms of the $R$-superalgebra $\Km_{10}\otimes_\FF R$, i.e., the group of $R$-linear isomorphisms preserving the multiplication and the grading), with the natural definition on morphisms.

It turns out that $\AAut\bigl(\Km_{10}\bigr)$ is isomorphic to a semidirect product $\bigl(\SSL_2\times\SSL_2\bigr)\rtimes \Cm_2$ (Theorem \ref{th:AAut_K10}), where $\Cm_2$ is the constant group scheme attached to the cyclic group of two elements (also denoted by $\Cm_2$). This extends the result in \cite{ELS07}.

Moreover, the canonical projection $\pi:\AAut\bigl(\Km_{10}\bigr)\rightarrow \Cm_2$ induces an isomorphism of pointed sets $H^1\bigl(\FF,\AAut\bigl(\Km_{10}\bigr)\bigr)\rightarrow H^1(\FF,\Cm_2)$ (Theorem \ref{th:pi*_iso}). This associates to the isomorphism class of any twisted form of $\Km_{10}$, the isomorphism class of a quadratic \'etale algebra over $\FF$, which allows to get easily the classification of twisted forms of $\Km_{10}$ (Corollary \ref{co:pi*_iso}). Twisted forms of $\Km_{10}$ were classified, in a completely different way, in \cite{RZ} and \cite{EO}. We believe that our approach is more natural.

A simple observation shows that $\bigl(\SSL_2\times\SSL_2\bigr)\rtimes \Cm_2$ is also the automorphism group-scheme of $\Km_3\times \Km_3$, or of $\Jm(W)\times \Jm(W)$, where $\Jm(W)$ is the Jordan superalgebra of the `superform' $(\,\mid\,)$ on $W=W\subuno$ above. That is, $\Jm(W)\subo=\FF 1$, $\Jm(W)\subuno=W$ and the multiplication is given by
\[
1x=x1=x,\qquad vw=(v\mid w)1,
\]
for any $x\in\Jm(W)$, $v,w\in W$. This observation will be used in the proof of the bijection $H^1\bigl(\FF,\AAut\bigl(\Km_{10}\bigr)\bigr)\simeq H^1(\FF,\Cm_2)$.

Finally, given an abelian group $G$, a $G$-grading on a superalgebra $\cA$ corresponds to a homomorphism of group schemes
\[
G^D\longrightarrow \AAut(\cA),
\]
where $G^D$ is the Cartier dual to the constant group scheme given by $G$. (See \cite{EKmon} for the basic definitions and facts on gradings.)

The classification of $G$-gradings up to isomorphism on $\Km_3\times \Km_3$ (or $\Jm(W)\times\Jm(W)$) is an easy exercise (Proposition \ref{pr:gradingsK3K3}), and it follows at once from this the classification of $G$-gradings, up to isomorphism, on $\Km_{10}$ (Theorem \ref{th:gradings_K10}), thus extending and simplifying widely the results in \cite{CDM10}.

\bigskip


\section{The group-scheme of automorphisms}

Consider the two-dimensional vector space $W=(\Km_3)\subuno$, which is endowed with the nonzero skew-symmetric bilinear form $(\,\mid\,):W\times W\rightarrow \FF$. The special linear group-scheme $\SSL(W)$ coincides with the symplectic group scheme $\SSP(W)$, which is, up to isomorphism, the group-scheme of automorphisms of $\Km_3$ (or of $\Jm(W)$).

The constant group scheme $\Cm_2$ acts on $\SSL(W)\times \SSL(W)$ by swapping the arguments, and hence we get a natural semidirect product $\bigl(\SSL(W)\times\SSL(W)\bigr)\rtimes\Cm_2$. The group isomorphism $\Phi$ in \cite[p.~3809]{ELS07} extends naturally to a homomorphism of affine group-schemes:
\begin{equation}\label{eq:Phi}
\begin{split}
\Phi:\bigl(\SSL(W)\times\SSL(W)\bigr)\rtimes\Cm_2&\longrightarrow \AAut\bigl(\Km_{10}\bigr)\\
(f,g)\qquad&\mapsto\quad\Phi_{(f,g)}:\begin{cases}\  1\ \mapsto \ 1,\\
     a\otimes a\mapsto a\otimes a,\\
     v\otimes a\mapsto f(v)\otimes a,\\
     a\otimes v\mapsto a\otimes g(v),\\
     v\otimes w\mapsto f(v)\otimes g(w),\end{cases}\\
     \text{generator of $\Cm_2$}&\mapsto\quad \tau:\begin{cases}
            \ 1\ \mapsto\ 1,\\
             x\otimes y\mapsto (-1)^{xy}y\otimes x,\end{cases}
\end{split}
\end{equation}
for any $R$ in $\AlgF$, $f,g\in \SSL(W)(R)$ (i.e., $f,g\in \End_R(W\otimes_\FF R)\simeq M_2(R)$ of determinant $1$), $v,w\in W_R\bydef W\otimes_\FF R$, $x,y\in (\Km_3)_R$. (Note that a representation of a constant group scheme $\rho:\Gm\rightarrow \GGL(V)$ is determined by its behavior over $\FF$: $\rho_\FF:\Gm\rightarrow GL(V)$.)

\begin{theorem}\label{th:AAut_K10}
$\Phi$ is an isomorphism of affine group-schemes.
\end{theorem}
\begin{proof}
If $\overline{\FF}$ denotes an algebraic closure of $\FF$, then $\Phi_{\overline{\FF}}$ is a group isomorphism \cite[Theorem 3.3]{ELS07}. Since $\bigl(\SSL(W)\times\SSL(W)\bigr)\rtimes\Cm_2\simeq \bigl(\SSL_2\times\SSL_2\bigr)\rtimes\Cm_2$ is smooth, it is enough to prove that the differential $\textup{d}\Phi$ is bijective (see, for instance \cite[Theorem A.50]{EKmon}). The Lie algebra of $\bigl(\SSL(W)\times\SSL(W)\bigr)\rtimes\Cm_2$ is $\frsl(W)\times \frsl(W)$, while the Lie algebra of $\AAut\bigl(\Km_{10}\bigr)$ is the even part of its Lie superalgebra of derivations, which is again, up to isomorphism, $\frsl(W)\times\frsl(W)$ identified naturally with a subalgebra of $\End_\FF\bigl(\Km_{10}\bigr)$ (see \cite[Theorem 2.8]{BE02}). Moreover, with the natural identifications, $\textup{d}\Phi$ is the identity map.
\end{proof}

The last result in this section is the simple observation that $\bigl(\SSL(W)\times\SSL(W)\bigr)\rtimes\Cm_2$ is also the group-scheme of automorphisms of the Jordan superalgebra $\Km_3\times\Km_3$ and $\Jm(W)\times \Jm(W)$. This will be instrumental in the next section. Its proof is straightforward.

\begin{proposition}\label{pr:Psis}
The natural transformations defined by:
\[
\begin{split}
\Psi^1:\bigl(\SSL(W)\times\SSL(W)\bigr)\rtimes\Cm_2&\longrightarrow \AAut\bigl(\Km_3\times\Km_3\bigr)\\
(f,g)\qquad&\mapsto\quad\Psi^1_{(f,g)}:\begin{cases}(a,0)\mapsto (a,0),\\
        (0,a)\mapsto (0,a),\\
        (v,w)\mapsto \bigl(f(v),g(w)\bigr),
     \end{cases}\\
     \text{generator of $\Cm_2$}&\mapsto\quad \tau: (x,y)\mapsto (y,x),
\end{split}
\]
for $v,w\in W$ and $x,y\in \Km_3$,
and
\[
\begin{split}
\Psi^2:\bigl(\SSL(W)\times\SSL(W)\bigr)\rtimes\Cm_2&\longrightarrow \AAut\bigl(\Jm(W)\times\Jm(W)\bigr)\\
(f,g)\qquad&\mapsto\quad\Psi^2_{(f,g)}:\begin{cases}(1,0)\mapsto (1,0),\\
        (0,1)\mapsto (0,1),\\
        (v,w)\mapsto \bigl(f(v),g(w)\bigr),
     \end{cases}\\
     \text{generator of $\Cm_2$}&\mapsto\quad \tau: (x,y)\mapsto (y,x),
\end{split}
\]
for $v,w\in W$ and $x,y\in \Jm(W)$, are isomorphisms of group-schemes.
\end{proposition}

\bigskip


\section{Twisted forms}

The set of the isomorphism classes of twisted forms of $\Km_{10}$ is identified with the pointed set $H^1\bigl(\FF,\AAut\bigl(\Km_{10}\bigr)\bigr)$. (This is $H^1\bigl(\overline{\FF}/\FF,\AAut(\Km_{10}\bigr)\bigr)$ in the notation of \cite{Wat}.)

The key to the classification of twisted forms of $\Km_{10}$ is the following result:

\begin{theorem}\label{th:pi*_iso}
The canonical projection $\pi:\bigl(\SSL(W)\times\SSL(W)\bigr)\rtimes\Cm_2\rightarrow \Cm_2$ induces a bijection of pointed spaces $\pi_*:H^1\bigl(\FF,\AAut(\Km_{10}\bigr)\bigr)\rightarrow H^1(\FF,\Cm_2)$.
\end{theorem}
\begin{proof}
 First, the natural section $\iota:\Cm_2\rightarrow \bigl(\SSL(W)\times\SSL(W)\bigr)\rtimes\Cm_2$ satisfies $\pi\circ\iota=\id$, so $\pi_*\circ\iota_*=\id$, and $\pi_*$ is surjective.
 
Second, the short exact sequence
\[
1\longrightarrow \SSL(W)\times\SSL(W)\longrightarrow\bigl(\SSL(W)\times\SSL(W)\bigr)\rtimes\Cm_2
\longrightarrow \Cm_2\longrightarrow 1
\]
induces an exact sequence (see \cite[\S 18.1]{Wat}) in cohomology:
\begin{multline*}
1\rightarrow \SL(W)\times\SL(W)\rightarrow \bigl(\SL(W)\times\SL(W)\bigr)\rtimes\Cm_2\rightarrow \Cm_2\\  \rightarrow H^1(\FF,\SSL(W)\times\SSL(W))\rightarrow 
    H^1\bigl(\FF,\bigl(\SSL(W)\times\SSL(W)\bigr)\rtimes\Cm_2\bigr)\rightarrow 
    H^1(\FF,\Cm_2).
\end{multline*}
Since $H^1(\FF,\SSL_2)$ is trivial, this gives an exact sequence
\[
1\longrightarrow H^1\bigl(\FF,\bigl(\SSL(W)\times\SSL(W)\bigr)\rtimes\Cm_2\bigr)\rightarrow 
    H^1(\FF,\Cm_2),
\]
but, in principle, this does not prove that $\pi_*$ is a bijection. It just says that the only element in $H^1\bigl(\FF,\bigl(\SSL(W)\times\SSL(W)\bigr)\rtimes\Cm_2\bigr)$ that is sent by $\pi_*$ to the trivial element in $H^1(\FF,\Cm_2)$ is the trivial element.

Through the isomorphisms in the previous section, it is enough to prove that $\pi$ induces a bijection of pointed sets $\pi_*:H^1\bigl(\FF,\AAut\bigl(\Jm(W)\times\Jm(W)\bigr)\bigr)\rightarrow H^1(\FF,\Cm_2)$.

Write $\cJ=\Jm(W)\times \Jm(W)$ and identify $\bigl(\SSL(W)\times\SSL(W)\bigr)\rtimes\Cm_2$ with $\AAut(\cJ)$ by means of $\Psi^2$ in Proposition \ref{pr:Psis}. Then $\cJ\subo=\FF\times\FF$, and $\pi$ corresponds to the restriction map $\AAut(\cJ)\rightarrow \AAut(\cJ\subo)$, and the induced map in cohomology $\pi_*:H^1(\FF,\AAut(\cJ))\rightarrow H^1(\FF,\AAut(\cJ\subo))$ corresponds to the map that takes the isomorphism class of a twisted form $\cH$ of $\cJ$ to the isomorphism class of $\cH\subo$. If $\cH\subo$ is trivial ($\cH\subo\simeq \FF\times\FF$) then, by the above, $\cH$ is isomorphic to $\cJ$. Otherwise, $\cH\subo$ is a quadratic separable field extension of $\FF$, so that $\cH\subuno$ is a two-dimensional vector space over $\cH\subo$, and the multiplication $\cH\subuno\times\cH\subuno\rightarrow \cH\subo$ is $\cH\subo$-bilinear and anticommutative, and hence uniquely determined. It turns out that $\cH$ is isomorphic to $\Jm(W)\otimes_\FF\cH\subo$. This proves the injectivity of $\pi_*$.
\end{proof}

\begin{corollary}\label{co:pi*_iso}
Any twisted form of $\Km_{10}$ splits after a quadratic separable field extension, and for any such field extension $\KK$ of $\FF$, there is a unique (up to isomorphism) nontrivial twisted form of $\Km_{10}$ which splits over $\KK$.
\end{corollary}

If $\KK$ is a quadratic separable field extension of $\FF$, the unique nontrivial twisted form of $\Km_{10}$ that splits over $\KK$ corresponds to the nontrivial element in $H^1(\KK/\FF,\Cm_2)$, and hence if the Galois group $\textrm{Gal}(\KK/\FF)$ is generated by $\sigma$ ($\sigma^2=\id$), this twisted form is given by
\begin{equation}\label{eq:twist_K}
\begin{split}
\Km_{10,\KK}&=\{X\in \Km_{10}\otimes_\FF\KK: (\tau\otimes\id)(X)=(\id\otimes\sigma)(X)\}\\
&=\{X\in\Km_{10}\otimes_\FF\KK: (\tau\otimes\sigma)(X)=X\},
\end{split}
\end{equation}
where $\tau$ is given in \eqref{eq:Phi}.

Hence, with $\KK=\FF(\alpha)$, $\alpha^2\in\FF$, $\sigma(\alpha)=-\alpha$, so if $\{u,v\}$ is a symplectic basis of $W=(\Km_3)\subuno$,
\begin{equation}\label{eq:K10K}
\begin{split}
\bigl(\Km_{10,\KK}\bigr)\subo&=\FF\textup{-span}\Bigl\langle 1\otimes 1, (a\otimes a)\otimes 1, (u\otimes u)\otimes\alpha,(v\otimes v)\otimes\alpha,\\[-3pt]
  &\hspace*{1.2in} (u\otimes v-v\otimes u)\otimes 1,(u\otimes v+v\otimes u)\otimes \alpha\Bigr\rangle,\\[3pt]
\bigl(\Km_{10,\KK}\bigr)\subuno&=\FF\textup{-span}\Bigl\langle (a\otimes x+x\otimes a)\otimes 1, (a\otimes x+x\otimes a)\otimes\alpha: x\in\{u,v\}\Bigr\rangle.
\end{split}
\end{equation}

\smallskip

\begin{remark} \null\quad
\begin{itemize}
\item Both $\Km_3$ and $\Jm(W)$ are rigid: any twisted form of any of these superalgebras is isomorphic to it. This is because $\AAut(\Km_3)\simeq\AAut(\Jm(W))\simeq \SSL(W)\simeq \SSL_2$.

\item The proof of Theorem \ref{th:pi*_iso} shows that the twisted forms of $\Km_3\times\Km_3$ (respectively $\Jm(W)\times\Jm(W)$) are, up to isomorphism, the superalgebras $\Km_3\otimes_\FF \KK$ (resp. $\Jm(W)\otimes_\FF\KK$), where $\KK$ is quadratic \'etale algebra over $\FF$, and any two such forms are isomorphic if and only if so are the corresponding \'etale algebras.
\end{itemize}
\end{remark}

\bigskip


\section{Gradings}

Given an abelian group $G$, a $G$-grading on a superalgebra $\cA=\cA\subo\oplus\cA\subuno$ is a decomposition into a direct sum of subspaces: $\cA=\bigoplus_{g\in G}\cA_g$, such that $\cA_g\cA_h\subseteq \cA_{gh}$ for any $g,h\in G$, and each homogeneous component is a subspace in the `super' sense of $\cA$: $\cA_g=(\cA_g\cap\cA\subo)\oplus(\cA_g\cap\cA\subuno)$. We will write $\deg(x)=g$ in case $0\ne x\in\cA_g$. Two $G$-gradings $\Gamma:\cA=\bigoplus_{g\in G}\cA_g$ and $\Gamma'=\bigoplus_{g\in G}\cA'_g$, are isomorphic if there is an automorphism $\varphi\in\Aut(\cA)$ such that $\varphi(\cA_g)=\cA'_g$ for any $g\in G$. 

A grading by $G$ is equivalent to a homomorphism of affine group schemes $G^D\rightarrow \AAut(\cA)$ (see \cite{EKmon}), and this shows that two superalgebras with isomorphic group-schemes of automorphisms have equivalent classifications of $G$-gradings up to isomorphism. Therefore, in order to classify gradings on $\Km_{10}$ it is enough to classify gradings on $\Km_3\times\Km_3$ (or $\Jm(W)\times\Jm(W)$), and this is straightforward.

Actually, fix a symplectic basis $\{u,v\}$ of $W=(\Km_3)\subuno$. 

\begin{df}\label{df:Gamma12}
Given an abelian group $G$, consider the following gradings ($e$ denotes the neutral element of $G$):
\begin{itemize}
\item For $g_1,g_2\in G$, denote by $\Gamma^1(G;g_1,g_2)$ the $G$-grading given by:
\[
\begin{split}
&\deg(x)=e\ \text{for any $x\in\bigl(\Km_3\times\Km_3)\subo$},\\
& \deg(u,0)=g_1=\deg(v,0)^{-1},\quad
\deg(0,u)=g_2=\deg(0,v)^{-1}.
\end{split}
\]

\item For $g,h\in G$ with $h^2=e\neq h$, denote by $\Gamma^2(G;g,h)$ the $G$-grading given by:
\[
\begin{split}
&\deg(a,a)=e,\quad\deg(a,-a)=h,\\ 
&\deg(u,u)=g=\deg(v,v)^{-1},\quad \deg(u,-u)=gh=\deg(v,-v)^{-1}.
\end{split}
\]
\end{itemize}
\end{df}

\begin{proposition}\label{pr:gradingsK3K3}
Any grading on $\Km_3\times\Km_3$ by the abelian group $G$ is isomorphic to either $\Gamma^1(G;g_1,g_2)$ or to $\Gamma^2(G;g,h)$ (for some $g_1,g_2$ or $g,h$ in $G$). 

Moreover, no grading of the first type ($\Gamma^1(G;g_1,g_2)$) is isomorphic to a grading of the second type ($\Gamma^2(G;g,h)$), and
\begin{itemize}
\item $\Gamma^1(G,g_1,g_2)$ is isomorphic to $\Gamma^1(G;g_1',g_2')$ if and only if the sets\\ $\{g_1,g_1^{-1},g_2,g_2^{-1}\}$ and $\{g_1',(g_1')^{-1},g_2',(g_2')^{-1}\}$ coincide.

\item $\Gamma^2(G;g,h)$ is isomorphic to $\Gamma^2(G;g',h')$ if and only if $h'=h$ and $g'\in\{g,gh,g^{-1},g^{-1}h\}$.
\end{itemize}
\end{proposition}
\begin{proof}
Any $G$-grading on $\cJ\bydef \Km_3\times\Km_3$ induces a $G$-grading on $\cJ\subo$, which is isomorphic to $\FF\times\FF$, and hence we are left with two cases:
\begin{enumerate}
\item The grading on $\cJ\subo$ is trivial, i.e., $\cJ\subo$ is contained in the homogeneous component $\cJ_e$. Then, with $W=(\Km_3)\subuno$, both $W\times 0=(a,0)\cJ\subuno$ and $0\times W=(0,a)\cJ\subuno$ are graded subspaces of $\cJ\subuno$. Hence we can take bases $\{u_i,v_i\}$ of $W$, $i=1,2$, such that $\{(u_1,0),(v_1,0),(0,u_2),(0,v_2)\}$ is a basis of $\cJ\subuno$ consisting of homogeneous elements. We can adjust $v_i$, $i=1,2$, so that $(u_i\mid v_i)=1$. If $\deg(u_i,0)=g_i$, $i=1,2$, the grading is isomorphic to $\Gamma^1(G;g_1,g_2)$.

\item The grading on $\cJ\subo$ is not trivial. Then there is an element $h\in G$ of order $2$ such that $\deg(a,a)=e$ and $\deg(a,-a)=h$. (Note that $(a,a)$ is the unity element of $\cJ\subo\,(\simeq \FF\times\FF)$, so it is always homogeneous of degree $e$.) As $\cJ\subo=(\cJ\subuno)^2$, there are homogeneous elements $(u_1,u_2),(v_1,v_2)\in \cJ\subuno$ such that $(u_1,u_2)(v_1,v_2)=(a,a)$. If $g=\deg(u_1,u_2)$, this grading is isomorphic to $\Gamma^2(G;g,h)$.
\end{enumerate}
The conditions for isomorphism are clear.
\end{proof}

Any grading $\Gamma^1(G;g_1,g_2)$ is a coarsening of the grading $\Gamma^1\bigl(\ZZ^2; (1,0),(0,1)\bigr)$, while any grading $\Gamma^2(G;g,h)$ is a coarsening of the grading $\Gamma^2\bigl(\ZZ\times\ZZ_2; (1,\bar 0),(0,\bar 1)\bigr)$, where $\ZZ_2=\ZZ/2\ZZ$. As an immediate consequence, we obtain the next result.

\begin{corollary}\label{co:fine_gradings_K3K3}
Up to equivalence, there are exactly two fine gradings on $\Km_3\times\Km_3$: $\Gamma^1\bigl(\ZZ^2; (1,0),(0,1)\bigr)$ and $\Gamma^2\bigl(\ZZ\times\ZZ_2; (1,\bar 0),(0,\bar 1)\bigr)$.
\end{corollary}

Proposition \ref{pr:gradingsK3K3} and Corollary \ref{co:fine_gradings_K3K3} have completely similar counterparts for $\Jm(W)\times\Jm(W)$.

\smallskip

In order to transfer these results to Kac's superalgebra $\Km_{10}$, take into account that $\Km_{10}$ is generated by its odd part, as $\bigl((\Km_{10})\subuno\bigr)^2=(\Km_{10})\subo$, so any grading is determined by its restriction to the odd part, and use the commutativity of the diagram
\[
\xymatrix{  
\AAut\bigl(\Km_3\times\Km_3\bigr) \ar@{^{(}->}[d]
&\bigl(\SSL(W)\times\SSL(W)\bigr)\rtimes \Cm_2\ar[l]_{\Psi^1\qquad} \ar[r]^{\null\quad\qquad\Phi}  
&\AAut\bigl(\Km_{10}\bigr)  \ar@{_{(}->}[d]\\
\GGL(W\times W)\ar[rr]^{\simeq\qquad}
&&\GGL\bigl((W\otimes a)\oplus (a\otimes W)\bigr) 
}
\]
where the vertical arrows are given by the restrictions to the odd parts, and the bottom isomorphism is given by the natural identification $W\times W\rightarrow (W\otimes a)\oplus (a\otimes W)$, $(v,w)\mapsto v\otimes a + a\otimes w$.

Thus Definition \ref{df:Gamma12} transfers to $\Km_{10}$ as follows:

\begin{df}\label{df:K10}
Given an abelian group $G$, consider the following $G$-gradings on $\Km_{10}$:
\begin{itemize}
\item For $g_1,g_2\in G$, denote by $\Gamma^1_{\Km_{10}}(G;g_1,g_2)$ the $G$-grading determined by:
\[
 \deg(u\otimes a)=g_1=\deg(v\otimes a)^{-1},\quad
\deg(a\otimes u)=g_2=\deg(a\otimes v)^{-1}.
\]

\item For $g,h\in G$ with $h^2=e\neq h$, denote by $\Gamma^2_{\Km_{10}}(G;g,h)$ the $G$-grading determined by:
\[
\begin{split}
&\deg(u\otimes a+a\otimes u)=g=\deg(v\otimes a+a\otimes v)^{-1},\\ 
&\deg(u\otimes a-a\otimes u)=gh=\deg(v\otimes a-a\otimes v)^{-1}.
\end{split}
\]
\end{itemize}
\end{df}

\smallskip

And Proposition \ref{pr:gradingsK3K3} and Corollary \ref{co:fine_gradings_K3K3} are now transferred easily to $\Km_{10}$:

\begin{theorem}\label{th:gradings_K10}
Any grading on $\Km_{10}$ by the abelian group $G$ is isomorphic to either $\Gamma^1_{\Km_{10}}(G;g_1,g_2)$ or to $\Gamma^2_{\Km_{10}}(G;g,h)$ (for some $g_1,g_2$ or $g,h$ in $G$). 

No grading of the first type  is isomorphic to a grading of the second type, and
\begin{itemize}
\item $\Gamma^1_{\Km_{10}}(G;g_1,g_2)$ is isomorphic to $\Gamma^1_{\Km_{10}}(G;g_1',g_2')$ if and only if the sets\\ $\{g_1,g_1^{-1},g_2,g_2^{-1}\}$ and $\{g_1',(g_1')^{-1},g_2',(g_2')^{-1}\}$ coincide.

\item $\Gamma^2_{\Km_{10}}(G;g,h)$ is isomorphic to $\Gamma^2_{\Km_{10}}(G;g',h')$ if and only if $h'=h$ and $g'\in\{g,gh,g^{-1},g^{-1}h\}$.
\end{itemize}

Moreover, there are exactly two fine gradings on $\Km_{10}$ up to equivalence, namely 
$\Gamma^1_{\Km_{10}}\bigl(\ZZ^2; (1,0),(0,1)\bigr)$ and $\Gamma^2_{\Km_{10}}\bigl(\ZZ\times\ZZ_2; (1,\bar 0),(0,\bar 1)\bigr)$.
\end{theorem}

\smallskip

\begin{remark}
Any grading on $\Km_{10}$ by an abelian group $G$ extends naturally to a grading by either $\ZZ\times G$ or $\ZZ_2^2\times G$ on the exceptional simple Lie superalgebra $F(4)$, which is obtained from $\Km_{10}$ using the well-known Tits-Kantor-Koecher construction. However, not all gradings on $F(4)$ are obtained in this way.
\end{remark}

\end{document}